\documentclass[11pt]{amsart}
\usepackage{pstricks,pst-plot}

\newtheorem{theorem}{Theorem}[section]
\newtheorem{lemma}[theorem]{Lemma}
\newtheorem{corollary}[theorem]{Corollary}

\theoremstyle{definition}
\newtheorem{definition}[theorem]{Definition}
\newtheorem{example}[theorem]{Example}

\newtheorem{exercise}[theorem]{Exercise}

\newcommand{\bea}{\begin{eqnarray*}}
\newcommand{\eea}{\end{eqnarray*}}

\numberwithin{equation}{section}

\begin{document}
\title[Lectures]{INTRODUCTION TO COMPLEX DYNAMICS IN ONE DIMENSION}

%
%
\subjclass[2010]{}
\date{\today}

\maketitle

\centerline{\bf JOHN ERIK FORN\AE SS}

\bigskip

\centerline{\bf LECTURES YMSC JULY 16-August 8, 2013}

\bigskip

\centerline{YMSC, Yau Mathematical Sciences Center, Tsinghua University, Beijing}

\tableofcontents


\section{BASIC CONCEPTS }

\bigskip

The purpose of this course is to give an elementary introduction to complex dynamics in one dimension.
Much of the material comes from the textbook by John Milnor, [M]. For material about the equilibrium measure we refer to  the articles [FLM], [L]. Their work was generalized to higher dimension by Jean-Yves Briend and Julien Duval [BD]. We follow [BD] by restricting their presentation back to one dimension. This will be convenient for somebody who wants to continue into higher dimensional complex dynamics.

To start,  let $P(z)$ be a complex polynomial of degree $d \geq 2.$ Set $P^2(z)=P(P(z)), P^3(z)=P(P(P(z))), P^{n+1}(z)=P(P^n(z))$. We will investigate the behaviour of the sequence $\{P^n(z)\}$ as $n \rightarrow \infty.$\\

The plan of the course is the following:\\

\noindent Part 1: Introduction and basic concepts\\
Part 2: We will discuss invariant measures $\mu$. $\mu$ is a probability measure
and $\mu(P^{-1}(E))=\mu(E).$\\
Part 3: We will investigate some features of the Fatou set.\\

PART 1. Basics of Complex Dynamics in one variable.\\

\begin{definition}
$z_0$ is a fixed point if $P(z_0)=z_0.$
\end{definition}

 We classify fixed points according to the derivative
$\lambda:= P'(z_0).$\\
Attracting: $|\lambda|<1.$ If $\lambda=0$ it is called superattracting.\\
Repelling: $|\lambda|>1.$ \\
Neutral: $|\lambda|=1.$\\
The neutral case can be divided further: $\lambda^n=1$ for some positive integer $n$, is called Rationally neutral. If $\lambda^n \neq 1$ for all $n\geq 1$ is called irrationally neutral.

$z_0$ attracting:\\

Let $P(z)=P(z_0)+\lambda (z-z_0)+ \mathcal O(z-z_0)^2$. Assume that $z_0=P(z_0)=0.$ Then let $|\lambda|<\rho<1.$ If $|z|<\delta$, then $|\lambda (z-z_0)+ \mathcal O(z-z_0)^2|\leq \rho |z|$. Hence $|P|
\leq \rho |z-z_0|.$ Inductively, $|P^n(z)|\leq \rho^n|z|.$ Hence $P^n(z) \rightarrow z_0$ if $|z-z_0|<\delta.$

Let $U:=\{z\in \mathbb C; P^n(z) \rightarrow z_0\}.$ This is called the basin of attraction of $z_0.$
Then $U$ is an open set.
\
The immediate basin of attraction is the connected component of $U$ which contains $z_0.$\\

Conjugation: \\

Let $\phi$ be a 1-1 analytic function defined in a neighborhood of $z_0$, $\phi(z_0)=0$.
We say that $Q$ is conjugate to $P$ if $\phi\circ P=Q\circ \phi.$ If $Q$ is simpler than $P$, then it is easier to describe  the dynamics near $z_0$ using $Q$ instead of $P.$
We have that $Q=\phi\circ P\circ \phi^{-1}$ and
$$Q^n=(\phi\circ P\circ \phi^{-1})\circ (\phi\circ P\circ \phi^{-1})
\circ \cdots \circ (\phi\circ P\circ \phi^{-1})=\phi
\circ P^n\circ \phi^{-1}.$$

\begin{theorem} (Koenigs)
If $0<|\lambda|<1$, then we can find $\phi$ so that $Q=\lambda z.$
\end{theorem}

\begin{proof}
We can suppose that $z_0=0.$ We will define $\phi=\lim_n \phi_n$ where
$\phi_n=P^n/\lambda^n.$ Note that if $\phi_n$ converges to $\phi$, then $\phi'(0)=1$, since
$\phi'_n(0)=1$ for all $n.$ 
We have that $$\phi_n\circ P=P^{n+1}/\lambda^n= \lambda P^{n+1}/\lambda^{n+1}=\lambda \phi_{n+1}.$$ When $n\rightarrow \infty$, we get $\phi(P)= \lambda \phi.$

It remains to show that $\{\phi_n\}$ converges on some neighborhood of $0.$ We show convergence in some neighborhood of $0.$ $|P(z)-\lambda z|\leq C|z|^2$, $C$ large, $|z
|<\delta.$  Hence $$|P(z)|\leq |\lambda||z|+ C|z|^2\leq |\lambda||z|+C|z|\delta\leq
(|\lambda|+C\delta)|z|.$$
We can assume that $|\lambda|+C\delta<\rho<1.$ Hence $|P^n(z)|\leq \rho^n |z|$

\bea
|\phi_{n+1}-\phi_n| & = & |P^{n+1}/\lambda^{n+1}-P^n/\lambda^n|\\
& = & |(P(P^n(z))-\lambda P^n(z))/\lambda^{n+1}  |\\
& \leq & C|P^n(z)|^2/|\lambda|^{n+1}\\
\eea

$$
|\phi_{n+1}-\phi_n|\leq C\left(\frac{(|\lambda|+C\delta)^2}{|\lambda|}\right)^n\rightarrow 0
$$

[We shrink $\delta$ so that $\left(\frac{|\lambda|+C\delta)^2}{|\lambda|}\right)<1.$]

\end{proof}

Next we deal with the case when $P'(z_0)=0.$ Assume $z_0=0.$ Then $P=a_pz^p+a_{p+1}z^{p+1}+\cdots, p\geq 2, a_p \neq 0.$ [In this case we allow $P$ to be any analytic function, not necesssarily a polynomial.]
In this case, we have a theorem by Boettcher.

\begin{theorem}
There is a local biholomorphism $w=\phi(z),\phi(0)=0$, conjugating $P$ to the map $Q(w)=w^p.$
($w$ is called a Boettcher coordinate)
\end{theorem}

\begin{proof}
If we conjugate first with $w'=cz$, we get the map
$w'\rightarrow z=w'/c \rightarrow a_p(w'/c)^p+\cdots \rightarrow ca_p(w'/c)^p+\cdots$\\

Choose $c$ so that $c^{p-1}=a_p$. Note that this is unique up to $d-1$ roots of unity.
Hence we may assume that $P(z)=z^p+...$ We can then write
$P(z)=z^p(1+\cdots)$ and therefore $P^{1/p}=z(1+\cdots)$. In fact $P(z)=z^pe^{\alpha_0(z)},$
$|z|<\delta.$ Here $\alpha_0(0)=0.$ We get
$$
P^2(z)=(z^pe^{\alpha_0(z)})^pe^{\alpha_0(z^pe^{\alpha_0(z)})}=
z^{p^2}e^{\alpha_1(z)}, |z|<\delta, \alpha_1(0)=0.$$
In general, $P^n(z)=z^{(p^n)}e^{\alpha_n(z)}, |z|<\delta, \alpha_n(0)=0.$

Define $\phi_n=(P^n)^{1/p^n}=z (1+\cdots).$ 

We get that
$$
\phi_n\circ P = (P^n)^{1/p^n}\circ P=(P^n\circ P)^{1/p^n}=
\left((P^{n+1})^{1/p^{n+1}}\right)^p= \phi_{n+1}^p.
$$

Suppose that $\phi_n\rightarrow \phi.$ Then $\phi\circ P=\phi^p.$ This gives the conjugation.
We show convergence:

\bea
\frac{\phi_{n+1}}{\phi_n} & = & \frac{(P^{n+1})^{1/p^{n+1}}}{(P^n)^{1/p^n}}\\
& = & \left(\frac{(P^{n+1})^{1/p}}{P^n}\right)^{1/p^n}\\
& = & \left(\frac{(P\circ P^{n})^{1/p}}{P^n}\right)^{1/p^n}\\
& = & \left(\frac{\phi_1\circ P^{n}}{P^n}\right)^{1/p^n}\\
& = & \left(\frac{P^n+(\mathcal O(P^{n}))^2}{P^n}\right)^{1/p^n}\\
& = & (1+\mathcal O(P^n))^{1/p^n}\\
\eea

On a small neighborhood of $0$, we have

\bea
|P(z)| & \leq & 2|z|^p\\
|P^2|(z) & \leq & 2 (2^p |z|^{p^2})\\
|P^n|(z) & \leq & 2^{p^n}|z|^{p^n}\\
\eea

So $\frac{\phi_{n+1}}{\phi_n} = 1+\mathcal O (1/p^n).$
Hence the product $\phi_n = \phi_1 \Pi_{j=1}^{n-1} \frac{\phi_{j+1}}{\phi_j}$ converges to some
limit $\phi.$

\end{proof}

{\bf Exercises}\\

\begin{exercise}
Let $P(z)=z^2.$ Show that $0$ is a superattracting fixed point. Find the basin of attraction and the immediate basin of attraction
\end{exercise}

\begin{exercise}
Let $P(z)=z^2+1.$ Use local coordinate $w=\phi(z)=1/z$ near $\infty.$ Show that $0$ is a
superattracting point for the conjugate map $Q(w).$
\end{exercise}

\begin{exercise}
Let $P_c(z)=z^2+c.$ Find the fixed points $z_f$ and the derivatives $P'(z_f).$ Show that $P_c$ always has a repelling fixed point if $c\neq 1/4.$
\end{exercise}

\newpage



Let $z$ be a point where $P'(z)=0.$ Then $z$ is called a critical point and $P(z)$ is called a critical value. There are at most $d-1$ finite critical points (in addition $\infty$ is a critical point).

\begin{theorem}
Let $z_0$ be an attracting fixed point. Then there is a critical point in the immediate basin of attraction.
\end{theorem}

The proof of the theorem will use one of the most important tools in complex dynamics:

\begin{theorem}
The Universal Covering of $\mathbb C\setminus \{0,1\}$ is the unit disc,
$\Delta=\{z\in \mathbb C, |z|<1\}.$ So in other words, there is a holomorphic function
$\pi:\Delta \rightarrow \mathbb C \setminus \{0,1\}$ so that for every open disc $D\subset \mathbb C \setminus \{0,1\}$ the set $\pi^{-1}(D)$ is a countable union of open
set $U_n$ and $\pi:U_n\rightarrow D$ is biholomorphic for each $n.$
\end{theorem}

Using this we prove the Theorem:

\begin{proof}
If $z_0$ is superattracting, then $z_0$ is a critical point. So we can suppose that $P'(z_0)=\lambda,$
$0<|\lambda|<1.$ We choose a small neighborhood of $z_0$, $U_1$ on which $P$ has an inverse $f$.\\

We can assume that $U_1$ is a disc. Then $U_2:=f(U_1)$ is also simply connected. Suppose there is no critical value in $U_2.$ Then $f$ extends to $U_2,$ $U_3=f(U_2)=f^2(U_1).$ Here $f^2$ is an inverse of
$P^2.$ If there is no critical value in $U_3$ we can keep extending $f$. If there is no critical point in the immediate basin of attraction, then we can keep going for all $n,$ and obtain
a sequence $U_n=f^{n-1}(U_1).$ We also see that $(f^n)'(z_0)=(1/\lambda)^n \rightarrow \infty.$ 
We have $f^n(z_0)=z_0.$ Pick a point $w_0\in \Delta, \pi(w_0)=z_0.$ We can find 
liftings $F_n: U_1 \rightarrow \Delta; f_n=\pi \circ F_n$ and $F_n(z_0)=w_0$. The derivatives $F_n'(z_0)\rightarrow \infty$. This is impossible.
 This implies that there cannot be two points $p,q$ in $\mathbb C$ so that all $U_n\subset \mathbb C\setminus \{p,q\}.$ Hence for large $n$, $U_n$ must equal $\mathbb C$ or the plane with one point removed. Since $P$ has critical points, the only possibility is the plane with one point removed. We might assume this is $0.$ But then $0$ is the only critical point, so the only possibility is that $P(z)=a z^d.$ This is however impossible, since this polynomial has no attracting fixed point in $\mathbb C$ except $0$ which is superattracting.
\end{proof}


{\bf THE JULIA SET}

Set $p\in \mathbb C$. If there is an open neighborhood $U(p)$ on which $\{P^n_{|U(z)}\}$ is a normal family into $\overline{\mathbb C}$, then we say that $p$ is in the Fatou set $F$. If not, we say that $p$
belongs to the Julia set $J.$ The Fatou set is open, and the Julia set is compact. Infinity belongs to the Fatou set.

\begin{theorem}
The Julia set is nonempty. 
\end{theorem}

\begin{proof}
Assume that $J$ is empty. Then $\{P^n\}$ is normal on all of $\overline{\mathbb C}.$
Hence we can find a limit $f=\lim_k P^{n_k}.$ Then $f$ is a holomorphic map from $\overline{\mathbb C}$ to itself. First case is that $f$ is constant, $f\equiv a$. Then we take a small disc $D$ around $a$. For all large
$k,$  each $P^{n_k}(\mathbb P^1)$ must be contained in $D.$ This is impossible since each iterate of $P$ is onto $\mathbb P^1.$ Hence $f$ is nonconstant. There are finitely many points mapped to the origin, with total multiplicity $m<\infty.$ But then we can also conclude that for large $k$ the same is true for all $P^{n_k}.$ But these are polynomials of degree $d^{n_k}$ so will have $d^{n_k}$ zeroes with multiplicity.
\end{proof}

\begin{definition}
A set $E\subset \overline{\mathbb C}$ is completely invariant for $P$ if
$P(E)=E$ and $P^{-1}(E)=E.$
\end{definition}

\begin{lemma}
If $P^{-1}(E)=E,$ then $E$ is completely invariant.
\end{lemma}

\begin{proof}
Suppose that $P^{-1}(E)=E.$ If $p\in E,$ then $p\in P^{-1}(E),$ hence $P(p)\in E.$ So
$P(E)\subset E.$ If $q\in E,$ then there exists $p\in P^{-1}(E)$ so that $P(p)=q$ Hence there exists
$p\in E$ for which $P(p)=q.$ Hence $E \subset P(E).$
\end{proof}

\begin{theorem}
The Fatou set and the Julia set are both completely invariant. 
\end{theorem}

\begin{proof}
It suffices to show that the Fatou set is completely invariant. Suppose that $p\in F$ and $P(q)=p.$
We can find a neighborhood $U(p)$ on which the iterates $P^n$ is normal. But then $P^{n+1}$
is normal on a neighborhood of $q$. Hence $P^{-1}(F)\subset F.$ Suppose that $q\in F.$
Then there exists a neighborhood $U(q)$ on which the iterates $P^n$ is a normal family.
Let $\epsilon>0.$ Then we can shrink $U$ so that $P^n(U)$ must have diameter at most $\epsilon$ for all $n.$  Since $P$ is an open mapping, there must then also exist a neighborhood $V(P(q))$ so that
the diameters of all sets $P^n(V)$ are at most $\epsilon.$ But then the iterates $P^n$ is a normal family on $V.$ Hence $P(q)\in F.$ Therefore $q\in P^{-1}(F), $ so $F \subset P^{-1}(F).$
\end{proof}

\begin{theorem}
$J(P)=J(P^N), F(P)=F(P^N)$ for any integer $N\geq 2.$ 
\end{theorem}

\begin{proof}
It suffices to prove this for the Fatou set. If $p\in F(P)$, then there is a neighborhood $U(P)$ on which the
iterates $P^n$ is a normal family. Hence also the iterates $P^{Nn}$ is a normal family. So $F(P)\subset F(P^N).$ Suppose that $p\in F(P^N).$ Then there is a neighborhood $U(p)$ on which the iterates
$P^{Nn}$ is a normal family. Take any subsequence $P^{n_k}$. By taking a thinner subsequence,
we can assume $n_k=Nm_k+\ell$. For an even thinner subsequence, we an assume that $P^{Nm_k}$
converges to a map $f.$ But then $P^{Nm_k+\ell}=P^\ell \circ P^{Nm_k} \rightarrow P^\ell \circ f$.
\end{proof}

EXCEPTIONAL SET:\\

Let $z\in J$ and $U=U(z)\subset \mathbb C.$ Then $\{P^n_{|U}\}$ cannot be a normal family. Let
$V(z,U)=\cup_n P^n(U)\subset \mathbb C.$ If $\mathbb C\setminus V$ contains two points, then the family would be normal, hence $V$ can omit at most one point. If we shrink $U$ only this point can be omitted. We say that this point belongs to the exceptional set, $E_z$, so $E_z$ is either $\infty$ or $\infty$ and one finite point $b$. Assume that $a\neq b$. Then $a=P^n(c)$ for some $n$ and some $c\in U.$ But then $P(a)\in V$, hence $P(a)\neq b.$ Therefore $P^{-1}(E_z)=E_z$, so the exceptional set is completely invariant. We can assume that if there is such a $b$, then $b$ is zero. Then the map $P$ must be on the form $az^d.$ We can conclude:

\begin{lemma}
Suppose that $P(z)=az^d$. Then the exceptional set for any point in the Julia set
consists of the two points $\{0,\infty\}.$ For any other polynomial the exceptional set
consists of $\infty$ only.
\end{lemma} 

[More precisely, the maps $P$ are $P(z)=a(z-z_0)^d+z_0.$]

\begin{theorem}
If $z$ is a nonexceptional point, then $J\subset \overline{\cup P^{-n}(z)}$
\end{theorem}

\begin{proof}
If $ w\in J$ and $U=U(w),$ then $z=P^n(\eta)$ for some $\eta \in U.$ So $\eta \in \cup P^{-n}(z).$
\end{proof}

\begin{theorem}
If $z\in J,$ then $J=\overline{\cup P^{-n}(z)}.$
\end{theorem}

\begin{proof}
Since $J$ is completely invariant it follows that $\overline{\cup P^{-n}(z)}\subset J.$
The other inclusion follows from the previous theorem since $z$ is non exceptional.
\end{proof}

\begin{theorem}
Suppose that $Z\subset J$ is nonempty and completely invariant. Then $\overline{Z}=J.$
\end{theorem}

\begin{proof}
Pick $z\in Z.$ Then $P^{-n}(z)\subset Z$. Hence $\overline{Z}=J.$
\end{proof}

 {\bf Exercises}\\

\begin{exercise}
Show that an annulus $A=\{0<a<|z|<b<\infty\}$ has a universal cover of the form
$S=\{0<y<1\}, z=x+iy.$ Find an explicit covering map $\pi: S\rightarrow A.$ Let $D=\{z\in A;x>0\}$.
Find the inverse image of $A$. Show that $\pi:U_n\rightarrow D$ is biholomorphic on any connected
component $U_n$ of $\pi^{-1}(A).$
\end{exercise}

\begin{exercise}
Show that the exceptional set of $P$ and $P^N$ is the same for any polynomial $P.$
\end{exercise}

\begin{exercise}
Let $P(z)$ be a polynomial of degree $d\geq 2.$ Use the conjugation $w=1/z$ to describe the polynomial as $Q(w)$ near infinity. Show that $Q(w)$ is conjugate to $R(t)=t^d$ near $w=0.$
\end{exercise}

\newpage



\begin{theorem}
Suppose that $D$ is a union of connected components of $F$ and suppose that $D$ is completely invariant. Then $J=\partial D.$
\end{theorem}

\begin{proof}
The boundary of $D$ is completely invariant and nonempty and contained in $J.$
\end{proof}

\begin{theorem}
The Julia set contains no isolated point.
\end{theorem}

\begin{proof}
Assume that $z_0\in J$ is isolated in $J.$ We divide into two cases. Suppose first that $z_0$ is not
a periodic point. Pick a neighborhood $U(z_0).$\\

 We will show that there is a $w\in U\setminus z_0$
with $w\in J.$ Consider a point $z_1, P(z_1)=z_0.$ Then $z_1\in J$ and $z_1\neq z.$ Then there must exist a point $w\in U$ so that some iterate $P^n(w)=z_1$ and hence $P^{n+1}(w)=z_0$. Since $z_0$ is not periodic, $w\neq z_0.$

Next assume that $z_0$ is a periodic point. Replacing $P$ by $P^N$ for some $N$ we can assume that
$P(z_0)=z_0.$ If $P^{-1}(z_0)=\{z_0\},$ then $z_0$ is a fixed critical point, so cannot be in the Julia set.
Hence we can find $z_1\neq z_0$ with $P(z_1)=z_0.$ Then $z_1$ is in the Julia set, and again there must  be as above points $w$ close to $z_0$ which are mapped under iteration to $z_1$. Since $z_0$ is fixed we must again have that $w\neq z_0.$ 
\end{proof}

We have shown before that there are at most $d-1$ attracting orbits 
(in addition to the attracting point at infinity).

Next we want to estimate the number of  neutral periodic orbits. The basic idea is the perturb the polynomial so at least about half of the neutral points become attracting. Then there can be at most $2(d-1)$ such points.

To perturb, we observe first that for a polynomial $P(z)=z^d$ there are no periodic orbits with $|(P^n)'|=1.$ Next let $P(z)=a_dz^d+\cdots$ be a polynomial with at least one  neutral periodic orbit.
We can conjugate with $w=cz$, which does not change the derivative at fixed or periodic points,
and then we can assume that $P(z)=z^d+\cdots.$
Let $R(z,t)=R_z(t)=(1-t)P+tz^d, z,t \in \mathbb C.$ For $t=0,$ this is the polynomial $P$ and for $t=1$ this is the polynomial $z^d.$ 
The condition that $z_0$ is periodic of order $m$ is that $R^m_t(z_0)=z_0.$ The condition that the multiplier is $1$ is that $\frac{\partial}{\partial z}R^m_t(z_0)=1.$
Let $Z_m$ denote the zero set of the equation $R^m(z,t)-z$ and $X_m$ the zero set of
$\frac{\partial}{\partial z}T^m(z,t)-1.$ These are possibly singular Riemann surfaces in $\mathbb C^2$.
For each $t$ there are $d^m$ respectively $d^m-1$ points $z$ on these varieties, counted with muliplicity. \\

We note that for the value $t=1$, there is no $z$ belonging to both sets.
This implies the collection of $t$ values where the curves intersect is locally finite.

Suppose now that for some value $t=c$ we have a $z_0$ which is periodic of order exactly $m$ and
with multipler $\lambda \neq 1. $ We can then use the implicit function theorem to describe the
solutions to $R^m(z,t)-z=0$. We get a unique solution $z=z(t)$ for $t$ close to $c$ and $z(c)=z_0.$
The graph is inside the Riemann surface $Z^m$ and we can continue along this curve, avoiding intersection points between $Z^m$ and $X^m$ and branch points of $Z^m$ over the t axis.
This can be continued to the value $t=1$. Also the multiplier $\lambda(t)$ will be analytic along this curve. 

Suppose we start at $t=0$ with a neutral periodic orbit, $|\lambda|=1, \lambda \neq 1.$ Since the multiplier at $t=1$ does not have modulus 1, the analytic function $\lambda(t)$ cannot be constant. Hence if we move away from $t=0$ in  $1/2$ of the directions, the value of $|\lambda|$ will be strictly less than one, so the periodic orbit $z(t)$ will be attracting.

Suppose we start at $t=0$ with a periodic point $z_0$ of order $m$ with $\lambda=1.$ In this case $Z^m$ might be singular at $(0,z_0).$ In this case, we can still parametrize $Z_m$ by
$t=\tau^k, z= z_0+\sum_{j \geq 1} a_j \tau^j$ in a neighborhood. [Puisseux series][There might be finitely many of these through $(0,z_0)$, just pick one.] The multiplier in this case is a holomorphic function $\lambda(\tau)$ with $\lambda(0)=1.$ Still it must be nonconstant because we can still anaytically continue to $t=1.$ Hence there will be half of the angles in $\tau$ where the orbit becomes attracting. This implies that at least half the directions in $t$ space will become attracting.

Suppose you have $2d$ neutral periodic orbits for $P.$ Then there must be some angle where
$d$ of them are attracting. This is impossible. So there are at most $2d-1$  neutral periodic orbits. Since in addition there might be $d$ orbits with $|\lambda|<1$ (including the point at infinity), 
we can conclude:

\begin{theorem}
A polynomial $P(z)$ of degree $d$ can have at most $3d-1$ periodic orbits which are {\bf not} repelling.
\end{theorem}

\begin{theorem}
The Julia set is the closure of the repelling periodic orbits.
\end{theorem}

\begin{proof}
Let $z\in J$ and $U(z)$ a neighborhood of $z.$ We want to show that there is a repelling periodic point in $U.$ Since $z$ is not isolated in $J$ and there are only finitely many nonrepelling periodic orbits,
we can assume that all periodic points in $U$ are repelling. Assume also that there are no repelling periodic points in $U.$ We can further move to another $z$ if necessary and assume there are no critical values in $U$. Then we can assume there are two preimages $z_1,z_2$ of $z$ and
that there are two inverses $f_1$ and $f_2$ of $P$ defined on $U$, $z_j=f_j(z).$ 
We can assume $f_j(U),U$ are three disjoint sets.

Define the functions $g_n(w)=\frac{P^n(w)-f_1(w)}{P^n(w)-f_2(w)}\frac{w-f_2(w)}{w-f_1(w)}.$
None of the 4 expressions used can vanish at any point in $U$. Hence the functions $g^n$ cannot take the value $0$ or $\infty$ on $U.$ Suppose that there is some value of $w$ and some $n$ for which
$g^n(w)=1.$ Then
$$
P^n(w)w-f_1(w)w-P^n(w)f_2(w)+f_1f_2=P^nw-f_2w-P^nf_1+f_1f_2
$$
so $(f_2-f_1)w= (f_2-f_1)P^n$ and hence $P^n(w)=w$ which is impossible since there are no
periodic points in $U.$ Hence the functions $g^n$ is a normal family on $U.$

This implies that $h_n=\frac{P^n-f_1}{P^n-f_2}$ is a normal family.
Since $P^n=\frac{f_2 h_n-f_1}{h_n-1}, $ it follows that $P^n$ is a normal family, a contradiction.

\end{proof}


\section{INVARIANT MEASURES}

ORIGINAL TEXT: 
 DEUX CARACTERISATIONS DE LA MEASURE DÉQUILIBRE D'UN ENDOMORPHISME
DE $P^k.$ by Jean-Yves Briend and Julien Duval.
But these lecture notes will give all details.\\

We will find probability measures that describe the dynamics on the Julia set. The key property is that of invariance. So let $\mu$ be a probability measure on $\overline{\mathbb C}.$ 
So for any Borel set $B$, $0\leq \mu(B) \leq 1=\mu(\overline{\mathbb C}).$
The property of invariance is that $\mu(P^{-1}(B))=\mu(B)$ for all Borel sets.
There are many such measures: For example, if $z_0$ is a fixed point, then the Dirac mass at $z_0$,
$\mu=\delta_{z_0}$ is such a measure: $\mu(B)=0$ if $z_0\notin B$, and $\mu(B)=1$ if $z_0\in B.$
Another example is $d\theta/(2\pi)$ on the unit circle and $P(z)=z^2.$
The second example has a stronger property: It is an equilibrium measure: Namely, if
we take a small arc of angle $\epsilon$ the preimage is two arcs {\bf with the same length},
$\epsilon/2.$  Note that the dirac mass at $\infty$ is such an invariant equilibrium measure.
For the map $P(z)=z^d$ the Dirac mass at the other exceptional point, $\delta_0$ is also
an equilibrium measure.

\begin{theorem}
Let $P$ be a polynomial of degree $d\geq 2.$ Then there is a unique equilibrium measure $\mu$
which gives no mass to the exceptional set. 
\end{theorem}

We will make some preparations. Let $|A|$ denote the area of a set $A.$

\begin{theorem} (Koebe distortion theorem)
Let $0<s<1$. Then there is a constant C so that if $f$ is any 1-1 holomorphic function $f:\Delta \rightarrow \Delta$, the unit disc
with $f(0)=0,$ then $\sup_{|z|=s} |f(z)|^2\leq C |f(\Delta)|$
\end{theorem}

\begin{proof}
Pick a number $t$, $s<t<1.$ For any $r, t\leq r \leq  1$ and any $z,|z|\leq s$ we have that
\bea
f'(z) & = & \frac{1}{2\pi i} \int_{|\zeta|=r } \frac{f'(\zeta)}{\zeta-z}d\zeta\\
& \Rightarrow &\\
|f'(z)| & \leq & \frac{1}{t-s}\int_0^{2\pi} |f'|r d\theta\\
& \Rightarrow & \\
|f'(z)| & \leq & \frac{1}{(t-s)(1-t)} \int_0^{2\pi} \int_t^1 |f'(\zeta)| rdrd\theta\\
& \leq & \frac{1}{(t-s)(1-t)} \int\int_{ |\zeta|<1} |f'(\zeta)| dxdy\\
& \leq & C \left(\int_{|\zeta|<1} |f'(\zeta)|^2 dxdy\right)^{1/2} \left( \int_{|\zeta|<1}  dxdy\right)^{1/2}\\
\eea
Here we have used the Cauchy-Schwartz inequality and the fact the real Jacobian of the holomorphic function $f(z)$ is $|f'(z)|^2.$ 

Integrating $f'$ from $0$ to $\{|z|=s\}$ we get the desired estimate for $f(z).$

\end{proof}


{\bf Exercises}\\

\begin{exercise}
Let $X=\{(t,z)\in \mathbb C^2; z^2-t^3=0\}$. Show that this set can be parametrized in the form
$t=\tau^k, z=f(\tau)$. Find the smallest possible $k$ and the function $f(\tau).$
\end{exercise}

\begin{exercise}
Let $P(z)=z^3.$ Show that the measures $\delta_0,\delta_\infty$ and $\mu=\frac{d\theta}{\pi}$ on the
unit circle are invariant probability measures. Show that  if $L$ is a small arc on the unit circle,
then $P^{-1}(L)$ consists of three arcs whose total length is the same as for $L$, and the three arcs have all the same length. 
\end{exercise}

\begin{exercise}
Let $P$ and $Q$ be conjugate, $\phi \circ P=Q\circ \phi.$ Show that if $z_0$ is a fixed point
for $P$, then $w_0=\phi(z_0)$ is a fixed point for $Q$ and that $Q'(w_0)=P'(z_0).$
\end{exercise}

\newpage



Let $C$ denote the critical points, and $V$ the image of the critical points.
Set $V_\ell:= \cup_{q=1}^\ell P^q(C)$, the critical values of $f^\ell$ and $V_\infty = \cup_{\ell \geq 1} V_\ell$ the postcritical set. 

\begin{lemma} (Lyubich)
Let $\epsilon >0.$ There exists $\ell=\ell(\epsilon)>0$ so that for any topological discs  $D\subset \subset \tilde{D}\subset \subset \mathbb C$ which do not intersect $V_\ell$ there are at least $(1-\epsilon)d^n$ inverse branches $g_i$ of $f^n $ on $\tilde{D}$ for $n>n(\epsilon, D,\tilde{D})$ sufficiently large
with images $\Delta_i^n=g_i(D)$ with diameter at most $c d^{-n/2}$, $c$ is independent of $n.$ 
\end{lemma}

\begin{proof}
We will fix a large  $\ell=\ell(\epsilon)$ below. 
We know that $f^\ell$ has $d^\ell$ well defined inverses of $\tilde{D}$ with images $\tilde{D}^\ell_i$.
At most $d$ of these discs can contain a point in $V$. All the other discs have $d$ preimages.
This creates at least $d(d^\ell-d)=d^{\ell+1}-d^2$ preimages of $f^{\ell+1}.$ 
Same way one gets at least $d(d^{\ell+1}-d^2-d)=d^{\ell+2}-d^3-d^2$ preimages of $f^{\ell+2}.$
We easily see that in general we get at least $d^{n}(1-\epsilon/2)$ preimages of $f^n$
if $\ell$ is chosen large enough. Now note that all the images must lie in a fixed set $\{|z|<R\}$ for all large $n$. Let $N$ be the number of discs with area $\geq \frac{2\pi R^2}{d^n\epsilon}.$ Then since they are disjoint, $N\frac{2\pi R^2}{d^n\epsilon}\leq \pi R^2.$ Hence $N<\frac{d^n\epsilon}{2}.$ So at least 
$d^n(1-\epsilon)$ of the discs must have area less than $\frac{2\pi R^2}{ d^{n}\epsilon}.$
The estimate on the diameter follows from the Koebe distortion theorem.
\end{proof} 

Let $x\in \overline{\mathbb C},$ For every $n\geq1$ we define 
a probability measure $\mu_{n,x}= \frac{1}{d^n} \sum_{z, P^n(z)=x} \delta_z$. Here we count $z$
with muliplicity, so there are $d^n$ points. In the case $x=\infty,$ we get $\mu_{n,\infty}=
\delta_\infty$. \\

We introduce a concept of convergence for measures:
A sequence of finite measures $\lambda_n$ converge {\bf weakly} to a finite measure $\lambda$
if for every continuous function $\phi$ on $\overline{\mathbb C}$ 
we have that $\int \phi d\lambda_n \rightarrow \int \phi d\lambda.$
Let $\mathcal{C}(\overline{\mathbb C})$ denote the space of continuous functions.
We can identify any measure with a point in ${\mathbb C}^{\mathcal C(\overline{\mathbb C})}$
where we map any measure $\mu$ to the point $\{\int \phi d\mu\}_\phi.$ The subset of probability measures is a compact subset. {We use the weak topology: Let $\phi_1,\dots,\phi_n$ be continuous functions and $U_1,\dots,U_n$ be open sets in ${\mathbb C}$. A basis for the topology is given by $\{z\}_\phi$ so that $z_{\phi_j}\in U_j$ for $j=1,\dots,n.$ In particular, any sequence of probability measures has a weakly convergent subsequence.

\begin{lemma}
If $\mu_{n,x}(y)\rightarrow 0$, then $\mu_{n,x}(P(y))\rightarrow 0.$
\end{lemma}

\begin{proof}
To find $\mu_{n,x}(z) $ we calculate the polynomial $P^n(z)-x$. For $z$ to have mass,
$P^n(z)-x$ must be zero at $z$. If the order of the zero is $r$, then the mass of
$z$ will be $r/d^n.$ Suppose that $\mu_{n,x}(P(y))$ does not go to zero when $n\rightarrow \infty$.
Then there must be some positive $\delta$ and a subsequence $n_k$ so that
$P^{n_k}(z)-x= \mathcal O((z-P(y))^{N_k})$ for $N_k>\delta d^{n_k}.$ Suppose next that $w\rightarrow y$. Then $|P(w)-P(y)|\leq C|w-y|.$ But then $|P^{n_k+1}(w)-x|=|P^{n_k}(P(w))-x|
\leq C' |P(w)-P(y)|^{N_k}\leq C'' |w-y|^{N_k}.$ This implies that $\mu_{n_k+1}(y)
\geq N_k/d^{n_k+1}\geq \frac{\delta}{d},$ so does not go to zero.
\end{proof}

\begin{corollary}
If $\mu_{n,x}(C)\rightarrow 0,$ then for any given $\ell,$ $\mu_{n,x}(V_\ell)\rightarrow 0$ when $n\rightarrow \infty.$
\end{corollary}

\begin{proof}
We apply the lemma $\ell$ times to each point in $C.$ Since there is a finite number of points in $C$
we are done.
\end{proof}

\begin{theorem}
Suppose that $x,y$ are two points so that $\mu_{n,x}(C),\mu_{n,y}(C)$ converge to $0$.
Then $\mu_{n,x}-\mu_{n,y}$ converges weakly to zero.
\end{theorem}

\begin{proof}
Pick $\epsilon>0.$ Let $\ell=\ell(\epsilon).$

Fix a continuous function $\phi, |\phi(z)|\leq 1, z\in \overline{\mathbb C}.$ Pick two points $z,t\in \mathbb C \setminus V_\ell.$
Let $\gamma$ be a curve consisting of two almost parallel straight lines from $z$ to $t$ avoiding $V_\ell$.
Let $D \subset \subset \tilde{D}$ be  thin topological discs around $\gamma$ which are disjoint from $V_\ell.$ For $n\geq n(z,t)$ large enough, $(1-\epsilon)
d^n$ of the preimages of the discs have diameter so small that $\phi$ varies at most $\epsilon$ on each of those discs. Let $D_i,i=1,\dots,[(1-\epsilon)d^n]$ be a counting of such discs.
Let $z_i,t_i$ be the corresponding preimages of $z,t\in D_i.$ For the remaining preimages
$z_i,t_i, i=[(1-\epsilon)d^n]+1, \dots,d^n$ list the preimages arbitrarily.
We get

\bea
|\int \phi d\mu_{n,z}-\int \phi d\mu_{n,t}| & = & 
\frac{1}{d^n}|\sum_{i=1}^{d^n} (\phi(z_i)-\phi(t_i))|\\
& \leq & \frac{1}{d^n}|\sum_{i=1}^{[(1-\epsilon)d^n]} (\phi(z_i)-\phi(t_i))|\\
& + & \frac{1}{d^n}|\sum_{i=[(1-\epsilon)d^n]+1}^{d^n} (\phi(z_i)-\phi(t_i))|\\
& \leq & \frac{1}{d^n}|\sum_{i=1}^{[(1-\epsilon)d^n]} \epsilon|+\frac{1}{d^n}|\sum_{i=[(1-\epsilon)d^n]+1}^{d^n} 2|\\
& \leq & 3 \epsilon\\
\eea

We next show that $\int \phi \mu_{n,x}-\int \phi \mu_{n,y}\rightarrow 0.$ This will prove
weak convergence of $\mu_{n,x}-\mu_{n,y}$ to $0.$

We know that $\mu_{n,x}(V_\ell),\mu_{n,y}(V_\ell)$ converge to $0$  when $n\rightarrow \infty.$

Next pick a large $m$ so that $\mu_{m,x}(V_\ell)+\mu_{m,y}(V_\ell) \leq \epsilon.$
We then fix $z_j\in P^{-m}(x), j=1,\dots,[(1-\epsilon)d^m], t_j\in P^{-m}(y)), j=1,\dots,[(1-\epsilon)d^m]$ in the complement of $V_\ell$ and label the remaining points in $P^{-m}(x),
P^{-m}(y)$ arbitrarily.

\bea
|\int \phi d\mu_{n,x}-\int \phi d\mu_{n,y}|& = &
|\frac{1}{d^n}\sum_{P^n(w_i)=x} \phi(w_i)-\frac{1}{d^n}\sum_{P^n(\eta_i)=y} \phi(\eta_i)|\\
& = & |\frac{1}{d^n}\sum_{P^m(z_j)=x}\sum_{P^{n-m}(w_{i,j})=z_j} \phi(w_{i,j})\\
& - & \frac{1}{d^n}
\sum_{P^m(t_j)=y}\sum_{P^{n-m}(\eta_{i,j})=t_j} \phi(\eta_{i,j})|\\
& = & |\frac{1}{d^m}\sum_{P^m(z_j)=x}\frac{1}{d^{n-m}}\sum_{P^{n-m}(w_{i,j})=z_j} \phi(w_{i,j})\\
& - & \frac{1}{d^m}
\sum_{P^m(t_j)=y}\frac{1}{d^{n-m}}\sum_{P^{n-m}(\eta_{i,j})=t_j} \phi(\eta_{i,j})|\\
& \leq & \frac{1}{d^m}\sum_{j=1}^{d^m}|\frac{1}{d^{n-m}}\sum_{P^{n-m}(w_{i,j})=z_j} \phi(w_{i,j})\\
& -  & \frac{1}{d^{n-m}}\sum_{P^{n-m}(\eta_{i,j})=t_j} \phi(\eta_{i,j})|\\
& \leq & \frac{1}{d^m}\sum_{j=1}^{[(1-\epsilon)d^m]}|\frac{1}{d^{n-m}}\sum_{P^{n-m}(w_{i,j})=z_j} \phi(w_{i,j})\\
& - & \frac{1}{d^{n-m}}\sum_{P^{n-m}(\eta_{i,j})=t_j} \phi(\eta_{i,j})|\\
& + & 2\epsilon\\
& = & \frac{1}{d^m}\sum_{j=1}^{[(1-\epsilon)d^m]}|\int \phi \mu_{n-m}(z_j) -\int \phi \mu_{n-m}(t_j)|\\
& + & 2\epsilon\\
& \leq & 3\epsilon+2\epsilon= 5 \epsilon\\
\eea

where the last inequality holds for all large enough $n.$ 

\end{proof}


{\bf Exercises}

\begin{exercise}
Show the following version of the Koebe distortion theorem:
\begin{theorem} (Koebe distortion theorem)
Let $0<s<1$. Then there is a constant C so that if $f$ is any 1-1 holomorphic function $f:(|z|<1) \rightarrow \mathbb C$
 then $\sup_{|z|,|w|\leq s} |f(z)-f(w)|\leq C \sqrt{\mbox{Area}(f(\Delta))}$
\end{theorem}

\end{exercise}

\begin{exercise}
Show that the sequence of measures $\{\delta_{1/n}-\delta_{-1/n}\}_n$ converges weakly to $0.$
\end{exercise}

\begin{exercise}
Let $P$ be a polynomial. Show that there exists a number $R>0$ so that $|P^{-n}(z)|<R$ for
any $z\in \mathbb C$ and all $n>n(z).$
\end{exercise}

\newpage



The Exceptional set $E$ of a polynomial is the largest finite set which is completely invariant.
It will always contain $\infty.$ For $P(z)=z^d$ also the origin belongs to the exceptional set.
But except for this case, only $\infty$ is an exceptional point.

\begin{lemma}
Suppose that $x$ is not in $E.$ Then $\mu_{n,x}(C)$ converges to $0.$
\end{lemma}

\begin{proof}
No preimage of $x$ can be in $E.$ If $z_0$ has only one preimage $w_0$ and is nonexceptional,
then $P(z)=a(z-w_0)^d+z_0$ and we see that $w_0$ has d preimages. If $z_0$ has
more than one preimage, then each preimage has multiplicity at most $d-1.$ Hence we
see that in any case the multiplicity of any preimage of $x$ under $P^{-2}$ has multiplicity at most
$d(d-1).$ Hence inductively, the multiplicity of any point in $P^{-n}(x)$ can be at most
$d(d(d-1))^{n/2}.$ Hence $\mu_{n,x}(C)$ goes to zero when $n \rightarrow \infty.$
\end{proof}

\begin{corollary}
If $x,y$ are points outside $E,$ then $\mu_{n,x}-\mu_{n,y}$ converges weakly to $0.$
\end{corollary}

The measures $\mu_{n,x}$ are examples of pull-backs of measures:
$$\mu_{n,x}= (P^n)^*(\delta_x)/d^n.$$ We will generalize this to general measures.
Let $\nu$ be a measure on $\overline{\mathbb C}.$ We define the pullback measure in the following way:
Let first $F$ be a set where $P:F \rightarrow P(F)$ is 1-1. Then we define $P^*(\nu)(F)= \nu(F).$
If $c$ is critical point of $P$ of multiplicity $m$ then $P^*(\nu)(c)= m \nu(P(c)).$
This is enough to define the measure of all Borel sets.

From this definition we get that $P^*(\delta_x)(y)=m\delta_x(P(y))$. In other words $P^*(\delta_x)$ gives mass to the points $P^{-1}(x)$ in the same way as $\mu_{1,x}$ does. We see that:

\begin{lemma}
$P^*(\mu_{n,x})=d\cdot \mu_{n+1,x}.$
\end{lemma}

A invariant probability measure $\lambda$ is an equilibrium measure if it satisfies $P^*(\lambda)=
d\cdot \lambda.$ 

We are now ready to prove our main theorem, stated earlier.

\begin{theorem}
Let $P$ be a polynomial of degree $d\geq 2.$ Then there is a unique equilibrium measure $\mu$
which gives no mass to the exceptional set. 
\end{theorem}

The first step in the proof is to show existence. This uses the important concept of Cesaro means.
Let $\lambda_j$ denote a sequence of probability measures.
Define $\sigma_n:=\frac{1}{n}\sum_{j=1}^n \lambda_j.$

We apply this construction of the measures $\mu_{j,x}$ for any given $x$ which is not an exceptional point.

$\lambda_n(x)=\frac{1}{n} \sum_{j=1}^n \mu_{j,x}.$

\begin{lemma}
The measures $\lambda_{n+1,x}-\lambda_{n,x}$ and
$\lambda_{n+1}- P^*\lambda_n/d$ have mass at most $2/(n+1).$
\end{lemma}

\begin{proof}
For the first statement,

\bea
\lambda_{n+1,x}-\lambda_{n,x} & = & \frac{1}{n+1} \sum_{j=1}^{n+1} \mu_{j,x}-
\frac{1}{n} \sum_{j=1}^{n} \mu_{j,x}\\
& = & \frac{1}{n+1}  \mu_{n+1,x}+\left(\frac{1}{n+1}-\frac{1}{n}\right)\sum_{j=1}^n \mu_{j,x}\\
\eea

From this it follows that the mass is at most $1/(n+1)- n\left(\frac{1}{n+1}-\frac{1}{n}\right)=2/(n+1).$\\

For the second statement, we get:

\bea
\lambda_{n+1}-P^*\lambda_n/d & = & 
\frac{1}{n+1} \sum_{j=1}^{n+1} \mu_{j,x}-\frac{1}{n}
P^*\left(  \sum_{i=1}^n \mu_{i,x}    \right)/d\\
& = & 
\frac{1}{n+1} \sum_{j=1}^{n+1} \mu_{j,x}- \frac{1}{n}\sum_{i=1}^n \mu_{i+1,x} \\
& = & 
\frac{1}{n+1} \sum_{j=1}^{n+1} \mu_{j,x}- \frac{1}{n}\sum_{j=2}^{n+1} \mu_{j,x} \\
& = & 
\left(\frac{1}{n+1}-\frac{1}{n}\right)\sum_{j=2}^{n+1}\mu_{j,x}+ \frac{1}{n+1}\mu_{1,x}\\
\eea

From this we see that the total mass, positive or negative of $\lambda_{n+1}-P^*\lambda_n/d $
is at most $\frac{2}{n+1}.$ 

\end{proof}

Next we will need the concept of {\bf duality}.
Let $\phi:\overline{\mathbb C}\rightarrow \mathbb C$ be a continuous function.
We define the push-forward of the function, $P_*(\phi): \overline{\mathbb C}\rightarrow \mathbb C$
by $P_*(\phi)(z)= \sum_{P(w)=z} \phi(w)$. We count with multiplicity. Then $P_*(\phi)$ is also continuous.

\begin{example}
$P_*(1)= d$
\end{example}

\begin{lemma} (DUALITY)
$$
\int_{\overline{\mathbb C}} \phi P^*\nu= \int _{\overline{\mathbb C}} P_*(\phi) \nu.
$$
\end{lemma}

\begin{proof}
We can find a partition of unity $\{\chi_i\}_{i=1}^N$ on $\mathbb P^1$ so that $\chi_1$ has support
near the critical set and the other $\chi_i$ have support in sets where $P$ is 1-1.
For $i>1$ we get that $\int (\chi_i\phi)P^*\nu= \int (\chi_i\phi)_* \nu$. Near the critical set the dominating contributions are  the point masses there.
\end{proof}

\begin{lemma}
There exists a probability measure $\mu$ on $\mathbb P^1$ so that $\mu=P^*(\mu)/d.$
Moreover the measure has no mass on the exceptional set.
\end{lemma}

\begin{proof}
Let $\mu$ be the weak limit of some subsequence $\lambda_{n_k}$. By the first part of Lemma 9.6,
it follows that $\mu$ is also the weak limit of the subsequence $\lambda_{n_k+1}.$
From the second part of the same lemma, it follows also that $\mu$ is the weak limit of the
sequence $P^*(\lambda_{n_k})/d$. We only need to show that $P^*(\lambda_{n_k})/d$
converge weakly to $P^*\mu/d.$ For this we use duality.\\

Let $\phi$ be a continuous function:
\bea
\int \phi P^*(\lambda_{n_k})/d  &= & \int P_*(\phi) \lambda_{n_k}/d\\
& \rightarrow & \int P_*(\phi) \mu/d\\
& = & \int \phi P^*\mu/d\\
\eea

It remains to show that there is no mass on $E$. We recall that we started the construction using a point $x$ which is not an exceptional point. Note that if there is a finite point in the exceptional set,
then $P(z)=z^d$. If $x\neq 0$, then $P^{-n}(x) \rightarrow \{|z|=1\}.$ Hence the support of the measure $\mu$ is on the unit circle, so it does not give mass to $0$ (nor to $\infty$).
\end{proof}

To complete the proof of Theorem 9.4, we need to show that $\mu$ is the unique equilibrium measure.
We show a stronger result.

\begin{theorem}
There is a unique probablity measure $\mu$ on $\mathbb P^1$ such that
$\frac{P^*\mu}{d}=\mu$ and with no mass on the exceptional set. Moreover, for any probablity measure $\nu$ with no mass on $E,$
$$
\frac{(P^n)^*\nu}{d^n} \rightarrow \mu.
$$
In particular, $\mu_{n,x}=\frac{(P^n)^*\delta_x}{d^n}$ converges to $\mu$ if and only if
$x$ does not belong to $E.$
\end{theorem}

In fact the only part of the proof missing is to show that  for any probablity measure $\nu$ with no mass on $E,$
$$
\frac{(P^n)^*\nu}{d^n} \rightarrow \mu:
$$

In fact this implies uniqueness: If $\frac{P^*\nu}{d}=\nu,$
then $\nu= \frac{(P^n)^*\nu}{d^n}\rightarrow \mu,$ so $\nu=\mu.$

Let $\phi$ be any continuous function. We will show that $\int \phi\cdot (P^n)^*\nu/d^n
\rightarrow \int \phi d\mu.$ This proves weak convergence.

Define continuous functions $F_n(x)=\frac{1}{d^n}(P^n)_*(\phi)(x) = \frac{1}{d^n}\sum_{P^n(w)=x} \phi(w)=\int \phi(z) \mu_{n,x}(z)$ and $G_n(y)=\frac{1}{d^n}(P^n)_*(\phi)(y) = \frac{1}{d^n}\sum_{P^n(\eta)=y} \phi(\eta)=\int \phi(z) \mu_{n,y}(z)$
We know that the measures $\mu_{n,x}-\mu_{n,y}$ converge weakly to $0$ for any $x,y$
which are non exceptional points. It follows that the function $H_n(x,y)=F_n(x)-G_n(y)$ goes pointwise to $0.$ Hence $\int \int H_n(x,y) d\mu(x) d\nu(y) \rightarrow 0.$

We have that 

\bea 
\int \int F_n(x) d\mu(x) d\nu(y) & = & \int F_n(x) d\mu(x)\\
& = &
\int \frac{1}{d^n}(P^n)_*{\phi} d\mu(x)\\
& = & \int \phi\cdot \frac{ (P^n)^*\mu}{d^n} =\int \phi \mu\\
\eea

and 
\bea 
\int \int G_n(y) d\mu(x) d\nu(y) & = & \int G_n(y) d\nu(y)\\
& = & \int \frac{1}{d^n}(P^n)_*(\phi)(y) d\nu(y)=
\int \phi(y) \cdot \frac{(P^n)^*\nu}{d^n}.\\
\eea

 So we have shown that $\int \phi \cdot \frac{(P^n)^*\nu}{d^n}
\rightarrow \int \phi d\mu$.

{\bf  Ergodicity and mixing}

\begin{definition}
Let $\nu$ be a probability measure. 
We say that $\nu$ is mixing if for every pair of Borel sets $E,F$,
$\nu(E \cap P^{-n}(F)) \rightarrow \nu(E) \nu(F)$ when $n \rightarrow \infty.$
\end{definition}

\begin{lemma}
Let $\phi,\psi$ be two continuous functions. Then
$$\int \phi(z) \psi(P^n)(z) d\mu(z) \rightarrow (\int \phi d\mu )(\int \psi d\mu).$$
\end{lemma}

\begin{proof}
We use duality. Note that if $\nu$ is a Borel measure and $\lambda$ is continuous,
then $\lambda d\nu$ is a Borel measure.

First observe that if $\phi$ is a continuous function and $x$ is outside $E$,
then (*)
\bea
\frac{(P^n)_*\phi}{d^n}(x) & = & \frac{1}{d^n}\sum_{z\in P^{-n}(x)}\phi(z)\\
& = & \frac{1}{d^n} \int \phi(z) \mu_{n,x}(z) \\
& \rightarrow & \int \phi(z) d\mu(z)\\
\eea

Pick two continuous functions.
\bea
\int \phi(z) \psi(f^n)(z) d\mu(z) & = & \int \phi(z) \psi(f^n)(z) \frac{(f^n)^* d\mu}{d^n}\\
& = & \int \phi(z) \frac{(f^n)^* (\psi d\mu)}{d^n}\\
& = & \int \frac{(f^n)_*\phi(z)}{d^n}\psi(z) d\mu(z)\\
& \rightarrow & \int \phi(z) d\mu \int \psi d\mu\\
\eea

\end{proof}

We extend the lemma to the case when we have one bounded measurable function and one continuous function:

\begin{lemma}
Let $\sigma$ be a bounded measurable function and $\phi$ a  continuous function. Then
$\int \phi \sigma(P^n)(z) d\mu(z) \rightarrow (\int \phi d\mu )(\int \sigma d\mu).$
\end{lemma}

\begin{proof}
We note that $\sigma d\mu$ is a Borel measure.
We get that

$$
\int \frac{P^n_*\phi}{d^n} \sigma d\mu \rightarrow (\int \phi d\mu)(\int \sigma d\mu)\;{\mbox{by}}(*)
$$

But 
$$
\int \frac{P^n_*\phi}{d^n} \sigma d\mu = \int \phi \frac{(P^n)^*(\sigma d\mu)}{d^n}
= \int \phi \sigma(P^n) \frac{(P^n)^*(\mu)}{d^n}= \int \phi \sigma(P^n) d\mu.
$$

\end{proof}

Finally we extend the lemma to two measurable functions, $\lambda, \sigma:$

\begin{lemma}
Let $\lambda,\sigma$ be two bounded measurable functions. Then
$$\int \lambda(z) \sigma(P^n)(z) d\mu(z) \rightarrow (\int \lambda d\mu )(\int \sigma d\mu).$$
\end{lemma}

\begin{proof}
Pick $\epsilon>0.$ Then there exists a continuous function $\phi$ so that
$$\int |\lambda-\phi| d\mu<\epsilon.$$

\bea
|\int \lambda \sigma(P^n)(z) d\mu(z)-\int \lambda d\mu \int \sigma d\mu| & \leq &
|\int \phi \sigma(P^n)(z) d\mu(z)\\
& - & \int \phi d\mu \int \sigma d\mu|+2\epsilon \sup |\sigma|\\
& \leq & \epsilon+2\epsilon \sup |\sigma|, \; n\; \mbox{large}\\
\eea

\end{proof}

\begin{theorem}
The measure $\mu$ is mixing.
\end{theorem}

\begin{proof}
Let $E,F$ be two Borel sets. Set $\lambda:=\chi_E, \sigma:=\chi_F$, the characteristic functions.
Then $\mu(E \cap P^{-n}(F))= \int \lambda \sigma(P^n) d\mu\rightarrow (\int \lambda d\mu))
(\int \sigma d\mu)=\mu(E) \mu(F).$
\end{proof}

\begin{definition}
A probability measure $\nu$ is ergodic for $P$ if for every invariant Borel set $E$,
i.e. $\mu(P^{-1})E \setminus E), \mu(E \setminus P^{-1}(E)=0,$ we have that
$\mu(E)$ is zero or one.
\end{definition}

\begin{theorem}
The measure $\mu$ is ergodic.
\end{theorem}

\begin{proof}
We have that $\mu(E)= \mu(E \cap E)= \mu(E \cap P^{-n}(E))
\rightarrow (\mu(E))^2.$ So $\mu(E)$ can only be zero or 1.
\end{proof}

\newpage




\section{TOPICS ON FATOU SETS}

We started the course with studying the Fatou set, attracting basins, including superattracting basins.
Then we moved to study the Julia set. The key there is to study invariant measures.
We now return to study the Fatou set for the rest of the course.

We will investigate Siegel discs. This will use important techniques from dynamics, socalled small denominators. Text:  Chapter VI of Carleson-Gamelin.

Let $\lambda,|\lambda|=1$ be a complex number. We say that $\lambda$ is diophantine if the following holds:
There exist constants $c>0$ and $\mu>1$ so that for all integers $n\geq 1$ we have
$$
|\lambda^n-1| \geq \frac{c}{n^{\mu}}.
$$

\begin{theorem}(Siegel)
Suppose $P$ is a polynomial with $P(0)=0$ and with $P'(0)=\lambda$ where $\lambda$ is diophantine. Then there is a holomorphic conjugation $\phi, \phi(0)=0, \phi'(0)=1$ so that
$\phi (P(z))=\lambda \phi(z)$ on a neighborhood of the origin.
\end{theorem}

\begin{proof}

For the proof we will instead construct $h(z)=\phi^{-1}(z).$ Then the functional equation becomes
$$
h(\lambda z)= P(h(z)).
$$
 where $h$ is biholomorphic near the origin and $h'(0)=1.$ The function $h$ will be constructed as a limit from an inductive process.

Let $f(z)= \lambda z+ \hat{f}(z).$ Let $h(z)=z+\hat{h}(z)$ where $\hat{f}$ and $\hat{h}$ vanish to second order. We then get the equation:

\bea
h(\lambda z) & = & f(h(z)) \\
h(\lambda z) & = & \lambda h(z) +\hat{f}(h(z))\\
\lambda z + \hat{h}(\lambda z) & = & \lambda z+\lambda \hat{h}(z)+\hat{f}(h(z))\\
\hat{h}(\lambda z)-\lambda \hat{h}(z) & = & \hat{f}(h(z))\; (**)\\
\eea
 
Next we introduce the inductive construction.

Let $\psi$ be a coordinate change in a neighborhood of the origin,
$\psi(z)= z+\hat{\psi}(z)$ where $\psi(z)=\mathcal O(z^2).$ The construction is to find
$\psi$ so that if 
$$\psi^{-1}\circ f\circ \psi=g(z)=\lambda z + \hat{g}(z) \; (***),
$$ then $\hat{g}$ is smaller than $\hat{f}.$ We use an approximation to (**).

$$
\hat{\psi}(\lambda z) -\lambda \hat{\psi}(z) = \hat{f}(z). (****)
$$

We replace $h(z)$ by $z$ in the right side of (**). Then we solve for $\hat{\psi}$ and estimate the corrsponding $\hat{g}.$ Let $\hat{f}(z)=\sum_{n=2}^\infty b_n z^n.$
Set $\hat{\psi}(z)= \sum_{j=2}^\infty a_j z^j$. We solve for the coefficients and investigate convergence below.

\bea
\sum_{j=2}^\infty a_j (\lambda z)^j-\lambda \sum_{j=2}^\infty a_j z^j & = & \hat{f}(z)=\sum_{j=2}^\infty b_j z^j\\
a_j(\lambda^j-\lambda) & = & b_j\\
a_j & = & \frac{b_j}{\lambda^j-\lambda}\\
\eea

We will make some assumptions for our estimates:

\noindent a) $\frac{1}{|\lambda^n-1|} \leq c_0 \frac{n^\mu}{\mu !}$\\
b) $|\hat{f}'(z)| <\delta, |z|<r$\\

Here we have two parameters, $\delta, r$ and then after estimating we will get new values of these parameters for the next function. We will let $\eta$ be a parameter, $0<\eta<1/5$. We estimate $\hat{\psi}$ in $\{|z|< r(1-\eta)\}.$

Since $\hat{f}'(z)=\sum_{j=2}jb_j z^{j-1}$ we get from Cauchy estimates that 

$$
|b_j| \leq \frac{\delta}{jr^{j-1}}.
$$

We will use a calculation of an infinite sum:

\bea
\sum_{n=0}^\infty x^n & = & \frac{1}{1-x}, |x|<1\\
& \Rightarrow & \\
\sum_{n=\mu}^\infty n(n-1)\cdots (n-\mu+1) x^{n-\mu} & = & \frac{\mu !}{(1-x)^{\mu+1}}\\
& \Rightarrow & [n-\mu=j-1]\\
\sum_{j=1}^\infty j (j+1) \cdots (j+\mu-1) x^{j-1} & = & \frac{\mu !}{(1-x)^{\mu+1}}\\
& \Rightarrow & [x=1-\eta]\\
\sum_{j=1}^\infty j (j+1) \cdots (j+\mu-1) (1-\eta)^{j-1} & = & \frac{\mu !}{\eta^{\mu+1}}\\
\eea

Next we estimate $\hat{\psi}'$ including the convergence radius.
Suppose that $|z| <(1-\eta) r.$

\bea
|\hat{\psi}'(z)| & \leq & \sum_{j=2}^\infty j|a_j| |z|^{j-1}\\
& \leq & \sum_{j=2}^\infty \frac{j|b_j|}{|\lambda^j-\lambda|} ((1-\eta)r)^{j-1}\\
& \leq & \sum_{j=2}^\infty \frac{j\delta}{jr^{j-1}}\frac{1}{|\lambda^{j-1}-1|} r^{j-1}(1-\eta)^{j-1}\\
& \leq & \sum_{j=2}^\infty \delta c_0 \frac{(j-1)^\mu}{\mu !}(1-\eta)^{j-1}\\
& < & \frac{c_0 \delta}{\mu !}\sum_{j=1}^\infty j(j+1)\cdots (j+\mu-1)(1-\eta)^{j-1}\\
& = & \frac{c_0 \delta}{\eta^{\mu+1}}\\
\eea

We add an extra condition:

c) $c_0\delta < \eta^{\mu+2}$.

Then we can conclude that $|\hat\psi'|\leq \eta$ when $|z|<(1-\eta) r.$
Since $|\hat\psi'|\leq \eta$ and $\hat\psi(0)=0$ we get that $|\hat\psi|<\eta r$ on $|z|<(1-\eta)r.$
We therefore get estimates on $\psi=z+\hat{\psi}.$

\noindent 1) $\psi(\{|z|< (1-4\eta)r\})\subset \{|z|< (1-3\eta) r\}$\\
2) $\{|z|< (1-2\eta)r\}\subset \psi(\{|z|< (1-\eta)r\})$ and map has welldefined inverse there (argument principle)\\

Let $g=\psi^{-1}\circ f \circ \psi$ on $\{|z|<(1-4\eta)r\}.$ 
We see that $\psi(z)$ is in $\{|z|<(1-3\eta)r\}$ and hence if add another condition:\\

d) $\delta<\eta$\\

$f(\psi(z))\in \{|z|<(1-2\eta)r\}.$ But then $\psi^{-1}\circ f\circ \psi(z)\in \{|z|<(1-\eta)r\}.$

We next estimate the function $\hat{g}.$ Recall from (***)
\bea
\psi^{-1}\circ f\circ \psi(z) & = & \lambda z +\hat{g}(z)\\
& \Rightarrow & \\
f(\psi) & = & \psi (\lambda z+\hat{g}(z))\\
\lambda \psi+\hat{f}(\psi) & = & \lambda z+\hat{g}(z)+ \hat{\psi}(\lambda z+\hat{g}(z))\\
\lambda z +\lambda \hat{\psi}+\hat{f}(\psi) & = & \lambda z+\hat{g}(z)+ \hat{\psi}(\lambda z+\hat{g}(z))\\
\lambda \hat{\psi}+\hat{f}(\psi) & = & \hat{g}(z)+ \hat{\psi}(\lambda z+\hat{g}(z))\\
\hat{g}(z)+ \hat{\psi}(\lambda z+\hat{g}(z)) & = & \lambda \hat{\psi}+\hat{f}(z+\hat\psi) \\
& \Rightarrow & [{\mbox{using (***)}}]\\
\hat{g}(z)+ \hat{\psi}(\lambda z+\hat{g}) & = & \hat{\psi}(\lambda z)-\hat{f}+\hat{f}(z+\hat\psi) \\
\hat{g}(z) & = &  \hat{\psi}(\lambda z)-\hat{\psi}(\lambda z+\hat{g}(z)) +\hat{f}(z+\hat\psi)-\hat{f}(z)\\
\eea

We estimate $\hat{g}$ when $|z|<r(1-4\eta).$ Let $C$ denote the maximum value.

We get then:
\bea
C & \leq & \sup |\hat\psi'| C+\sup |\hat{f}(z+\hat\psi)-\hat{f}(z)|\\
& \leq & C\eta+ \delta \frac{c_0 \delta}{\eta^{\mu+1}} r\\
& \Rightarrow & \\
C & \leq & \frac{ c_0 \delta^2 r}{\eta^{\mu+1} }\frac{1}{1-\eta}\\
\eea

Using Cauchy estimates we get that for $|z|<(1-5\eta)r$,

$$
|\hat{g}'| \leq \frac{ c_0 \delta^2 }{\eta^{\mu+2} }\frac{1}{1-\eta}.
$$

It is the term $\delta^2$ that is the key to the construction.

We summarize what we have done sofar in the argument:
We start with a holomorphic function defined on $\{|z|<r\}$.
The function $|\hat{f}'|\leq \delta.$ It is replaced by $g$ on $\{|z|<r(1-5\eta)$ with the estimate
$|\hat{g}'|\leq  \frac{ c_0 \delta^2 }{\eta^{\mu+2} }\frac{1}{1-\eta}.$
We needed to assume that $0<\eta<1/5, c_0\delta <\eta^{\mu+2}, \delta<\eta.$
If we choose a $c_1$ small enough and require that $\eta<c_1$, then the first condition is satisfied,
and the third condition follows from the second one. So we only need to worry about the second condition. We introduce the constant $\tilde{c}_0= c_0 2^{\mu+2}(1-\eta).$
Then we can write 
$$
|\hat{g}'|\leq  \frac{ \tilde{c}_0 \delta^2 }{\eta^{\mu+2} }\frac{1}{2^{\mu+2}}.
$$

Now we will describe the inductive data. We start with $\eta_0,\delta_0$ satisfying the conditions
and define inductively:

Choose first a polynomial $P$ with $P(0)=0$, $P'(0)=\lambda$ diophantine$\eta_0,\delta_0$ 
small.\\

Inductively define

\bea
r_{n+1} & = &r_n(1-5\eta_n)\\
\eta_{n+1} & = & \eta_n/2\\
\delta_{n+1} & = & c_0 \delta_n^2 \eta_n^{-\mu-2}2^{-\mu-2}\\
\eea

We verify that $c_0\delta^n \leq \eta_n^{\mu+2}.$

We then get inductively defined maps $\{\psi_n,g_n\}.$ 
Here $g_0=P$ and $g_n= \psi_n^{-1} \circ g_{n-1}\circ \psi_n$ 
Then the functions $g_n$ is a normal family on a disc with positive radius and
taking limits we get the theorem.

\end{proof}

\newpage



We have shown that if $0$ is a neutral fixed point with multiplier $\lambda$ with $\lambda$ diophantine, then the map is conjugate to rotation. We will now give an example where $\lambda$ is not diophantine, and not a root of unity, and the map is not conjugate to rotation.

\begin{theorem}
There exists a $\lambda=e^{2\pi i \theta}$ which is not a root of unity, so that there is no polynomial $P$ of degree $d\geq 2$ with $P(0)=0, P'(0)=\lambda$ and $P$ is conjugate with rotation by $\lambda$.
\end{theorem}

We start with a lemma.

\begin{lemma}
Suppose that $P(z)=z^d+\cdots+\lambda z$ and that there is a biholomorphic map $h(w)$ defined on $\{|w|<2r\}$ such that $h(0)=0, h'(0)=1$ and $P= h(\lambda h^{-1}(z))$ and $h(\{|w|<2r\}) \supset \{|z|<r\}$. Then for every integer $n\geq 1$, $r^{d^n-1} \leq |\lambda^n-1|.$
\end{lemma}

\begin{proof}
Let $\{z_j\}_{j=1}^{d^n}$ denote the fixed points of $P^n$ counted with multiplicity. Then
$P^n(z)-z=  z^{d^n}+\cdots +(\lambda^n-1) z=\Pi_{j=1}^{d^n} (z-z_j)$ Let us say that $z_{d^n}=0.$
Then $\Pi_{j=1}^{d^n-1} (z-z_j) = z^{d^n-1}+\cdots +(\lambda^n-1).$ 
Looking at the constant term on both sides, we see that
$|\Pi_{j=1}^{d^n-1}z_j| = |\lambda^n-1|.$
Note that if $0<|z|<r$, then $h(\lambda^n h^{-1}(z))-z$ has no zero because of the irrationality of $\lambda.$ So
all $|z_j|\geq r,j=1,\dots, d^n-1.$
It follows that $r^{d^n-1}\leq \Pi_{j=1}^{d^n-1} |z_j| \leq |\lambda^n-1|.$ 
\end{proof}

To prove the theorem, we find a $\lambda$ which is not a root of unity, but where
there is no $r>0$ so that $r^{d^n-1} \leq |\lambda^n-1|$ for all $n.$

We will define $\lambda=e^{2\pi i \theta}$ where
$\theta$ has the form $\theta=\sum_{k=1}^\infty \frac{1}{2^{q_k}}$ for integers
$1<q_1<q_2<\cdots$ and they grow rapidly. It is easy to see that such a number cannot be rational.

Then $\lambda^{2^{q^{_\ell}}}=e^{2\pi i \sum_{k=1}^\infty \frac{2^{q_\ell}}{2^{q_k}}}=
e^{2\pi i \sum_{k=\ell+1}^\infty \frac{2^{q_\ell}}{2^{q_k}}}.$
Hence we get $|\lambda^{2q^{\ell}}-1|\leq 4\pi \frac{2^{q_\ell}}{2^{q_{\ell+1}}}.$
We now choose $q_{\ell+1}$ so large that 
$$
4\pi \frac{2^{q_\ell}}{2^{q_{\ell+1}}}< \left(\frac{1}{\ell}\right)^{d^{2^{q_\ell}}}
d=2, \dots, \ell.
$$
Then for this $\lambda$ there is no polynomial of any degree for which there is a linearization
in any neighborhood of $0.$

This finishes the proof of the theorem.\\

\bigskip

We next show that almost all numbers $\lambda$ with modulus $1$ are diophantine.
Let $\mu>2.$ 

\begin{lemma}
Let $D$ consist of all $0<\theta<1$ such that $|\theta-p/q|<q^{-\mu}$ for infinitely many $q$.
Then the Lebesgue measure of $D$ is $0.$
\end{lemma}

\begin{proof}
Fix $q$. Then the condition that $|\theta-p/q|<q^{-\mu}$ for some $p$
is that $\theta\in <p/q-q^{-\mu}, p/q+q^{-\mu}>$ for some $p=1,\dots,q.$ This has measure
$2q^{-\mu}\cdot q$. Hence $D$ has measure at most $\sum_{q=k}^\infty 2q^{1-\mu}\leq Ck^{2-\mu}.$ Hence the measure of $D$ is $0.$ 
\end{proof}

Suppose next that $\theta\notin D.$
Then there exists a $c>0$ so that $|\theta-p/q|>c/q^{-\mu}$ for all rational $p/q.$
Hence $|e^{2\pi i \theta}-1|>c' q^{1-\mu}$ and $e^{2\pi i \theta}$ is diophantine.\\

\bigskip

Next we discuss the case when $\lambda$ is a root of unity.
The simplest case is when $\lambda=1$ and there is a nonzero quadratic term.

\begin{theorem}
Let $P(z)=z+z^2+\cdots$. Then there exists a disc $|z+\epsilon|<\epsilon$ on which the iterates
converge uniformly to $0.$ The origin is in the Julia set.
\end{theorem}

\begin{proof}
We conjugate using the map $w=u+iv=\phi(z)=-1/z.$ Then the polynomial $P$ is conjugate on $|z|<\delta$ with $Q(w)$ defined for $|w|>1/\delta.$
We get $Q(w)=\frac{-1}{\frac{-1}{w}+\frac{1}{w^2}+\cdots}=
\frac{w}{1-\frac{1}{w}+\cdots}=w+1+\mathcal O(\frac{1}{w})$.
It follows that if $r$ is large enough, then $Q(\{u> r\}) \subset \{u>r+1/2\}.$
In the $z$ coordinate, the set $\{u>s\}$ corresponds to the disc
$\{|z+\frac{1}{2s}|<\frac{1}{2s}\}.$
Hence for large $s$
$P(\Delta(-\frac{1}{2s},\frac{1}{2s}))\subset \Delta(-\frac{1}{2s+1},\frac{1}{2s+1})$
This proves the uniform convergence.
It remains to prove that $0$ is in the Julia set. If not there is a disc
$\Delta(0,r)$ contained in the Fatou set. But then there is a convergent subsequence $P^{n_k}$
on $\Delta(0,r).$ Then the limit holomorphic function must be identically $0$.
On the other hand it must have derivative $1$ at the origin. This is impossible.

\end{proof}

The theorem generalizes to other cases when the derivative is a root of unity.
First of all we can make a linear change so that $P(z)=z+az^2+\cdots$ becomes on the
form $z+z^2+\cdots.$ If $P(z)=z+z^{k+1}+\cdots$ for $k\geq 1$ the proof can be modified by using a singular change of coordinates.
Set $w=z^k$ on some sector of angle $\frac{2\pi}{k}.$
Then the conjugate map $Q(w)=\left( w^{1/k}+w^{(k+1)/k}+\cdots   \right)^k
=w(1+w+\cdots)^k=w+kw^2+\cdots.$ Then the previous argument will still work to
find an open set where the iterates converge to $0.$

The only missing case is when $\lambda^n=1$ for some integer $n>1.$ In this case
we replace $P$ by $P^n$. This does not change the Fatou set. And $(P^n)'(0)=1$ so the previous results apply, depending on what is the lowest order term
in $P^n=z+az^k+\cdots.$\\

\bigskip

Next we will prove a theorem by Denjoy-Wolff:

\begin{theorem}
Let $f$ be an analytic function, $f(D)\subset D$, with $D$ the unit disc.
Then either $f$ fixes a point $p$ in the disc and is a biholomorphism of the disc, or there is a point $\alpha$ in $\overline{D}$ so that $f^n(z)$ converges uniformly on compact subsets to $\alpha.$
\end{theorem}

We will give a proof of this theorem by Beardon.
The biholomorphic maps of the unit disc are of the form
$R(z)= e^{i\theta}\frac{z-z_0}{1-z\overline{z}_0}.$ 

There are two case, first when $f$ is a biholomorphic map with no ficed point in the interior, second when $f$ is not biholomorphic.
We discuss first the case when $f$ is a {\bf biholomorphic map with no fixed point in the interior of the unit disc}. Note that we can also think of $f$ as a biholomorphic map on the upper halfplane. This extends by reflection to the lower halfplane. It will be of the form $\frac{az+b}{cz+d}$ with real coefficients. When we solve for fixed points, we get conjugate answers. Since there are no solutions in the upper halfplane, they must be on the real line (including possibly infinity).
We can assume it is the point at infinity. Then the map is of the form
$f(z)=Az+B$ for real numbers $A,B, A>0.$ If $A>1$, $\infty$ is an attracting fixed point. If $A<1,$ the fixed point given by solving $Az+B=z$ is an attracting fixed point.
If $A=1,$ $B$ must be nonzero. Hence $f^n(z)=z+nB$ so all points converge to $\infty.$\\

We discuss the second and last case, $f$ {\bf is not a biholomorphism.}\\


We will need the Poincare metric. For $z\in D$, and $\xi$ a tangent vector, we define the
length of $\xi$ as $ds(\xi)=\frac{|\xi|}{1-|z|^2}.$ If $w=e^{i\theta} \frac{z-a}{1-\overline{a}z}$
is an automorphism of the disc, then this is an isometry in this metric.
This follows after calculating that $|dw/dz|=\frac{1-|w|^2}{1-|z|^2}.$
We can define the distance $\rho(z,w)$ by integrating over all curves from $z$ to $w$ and minimizing.

\begin{lemma}
If $f:D\rightarrow D$ is holomorphic, but not an automorphism, then
$\rho(f(z),f(w))< \rho(z,w)$ for any $z \neq w.$
\end{lemma}

\begin{proof}
Fix $z\in D.$ Let $T(\zeta)=\frac{\zeta+z}{1+\overline{\zeta}z}$ and
$S(\eta) = \frac{\zeta-f(z)}{1-\overline{f(z)}\eta}.$
Then $U:=S\circ f\circ T$ is a holomorphic map from the disc to itself. Since $T(0)=z$ and $S(f(z))=0$ it follows that $U(0)=0$. The map $U$ is not an automorphism, so
if $\tau \neq 0$, then $|U(\tau)|<|\tau|.$ Choose $\tau= T^{-1}(w).$ It follows then
that $|U( T^{-1}(w)) |<|T^{-1}(w)|.$ But then $\rho(0,U( T^{-1}(w))<\rho(0, T^{-1}(w)).$
Then $\rho(S\circ f \circ T(0), S\circ f\circ T (T^{-1}(w))<\rho(T(0),T(T^{-1}(w)).$
So $\rho(f(T(0),f(w))<\rho(z, w)$, the result follows.
\end{proof}

\newpage



\begin{lemma}
Suppose that $f:D  \rightarrow D$ is not an automorphism. Then either there is a fixed point $p\in D$ and all orbits in $D$ converges to $p$ or for every $z\in D, |f^n(z)| \rightarrow 1.$
\end{lemma}

\begin{proof}
We will show that if there is a point $q\in D$ such that $|f^n(q)|$ does not converge to 1, then
there is an attracting fixed point in $D.$ 
There must exist an $r<1$ and a subsequence $f^{n_k}$ so that $|f^{n_k}(q)|< r.$
By picking $r<s<1$ large enough, we can also assume that $|f^{n_k+1}(q)|<s.$
It follows that there exists a constant $\sigma<1$ so that 
$\rho(f^{n_k+1}(q),f^{n_k+2}(q))\leq \sigma \rho(f^{n_k})(q),f^{n_k+1}(q))$.
This implies by induction that 
$$
\rho(f^{n_k+1}(q),f^{n_k+2}(q) \leq \sigma^k \rho(f^{n_1}(q),f^{n_1+1}(q)).
$$
Therefore$\rho(f^{n_k+1}(q),f^{n_k+2}(q))\rightarrow 0.$ Taking an even thinner subsequence we can also assume that $f^{n_k+1}(q)\rightarrow p$ for a point in the unit disc. But then $\rho(p,f(p))=0$, which implies that $p$ is a fixed point. Since $f$ is not an automorphism, $p$ must be an attracting fixed point and all points in $D$ converge under iteration to $p.$
\end{proof}

We can now assume that $|f^n(z)|\rightarrow 1$ for any $z$ in $D.$
Let $\epsilon>0$ be small. Set $f_{\epsilon}(z)=(1-\epsilon) f(z).$ Then the image of $f_\epsilon$
is in $\{|z|<(1-\epsilon)\}.$ The above Lemma applies to show that $f_{\epsilon}$ has an attracting fixed point $z_\epsilon$. Suppose that there is a subsequence $\epsilon_k\rightarrow 0$
so that the sequence $z_{\epsilon_k}$ does not converge to the boundary. Then we can assume they converge to a point $p$ in $D$. But then this point must be a fixed point for $f$. This is impossible  because $|p|=|f^n(p)|$ goes to the boundary.

Set $T_\epsilon=\frac{z-z_\epsilon}{1-\overline{z}_\epsilon z}$ This maps $z_\epsilon$ to $0.$
Hence the map $T_\epsilon \circ f \circ T_\epsilon^{-1}$ conjugates $f$ with a map which has an attracting fixed point at $0.$ If we let $D_\epsilon=\{|T_\epsilon|<|z_\epsilon|\},$
then we see that $f_{\epsilon}(D_\epsilon)\subset D_\epsilon.$ Here $D_\epsilon$ is a disc with
$0$ as a boundary point and containing a point $z_\epsilon$ converging to the boundary.

Next, let $D'$ be a limit of a subsequence $D_{\epsilon_k}.$ Then $D'$ is a disc with $0$ as a boundary point. Moreover, there is another boundary point $\alpha$ which is on the unit circle.
If $q$ is an interior point of $D'$, then $q\in D_{\epsilon_k}$ for all large $k$. Hence
$(1-\epsilon_k)f(q)$ in $D_{\epsilon_k}$. But then $f(q)\in \overline{D}'.$ 
It follows that $f(\overline{D'}) \subset \overline{D}',$ so $f^n(\overline{D}')\subset
\overline{D}'.$ Therefore $\{f^n\}$ converges uniformly to $\alpha$ on compact subsets of
$D'.$ Since $\{f^n\}$ is a normal family on the whole unit disc, it follows that $\{f^n\}$ converges uniformly to $\alpha$ on compact subsets of $D.$

This finishes the proof of the Denjoy-Wolff theorem.\\

Our next topic is the {\bf Snail Lemma.} 

\begin{theorem}
Suppose that $P(z)$ is a polynomial of degree $d\geq 2.$ Let $P(0)=0$ and $P'(0)=e^{2\pi i \theta}$
where $\theta$ is irrational. Then there is no Fatou component $U$ so that $F(U)=U$ and
$f^n(z) \rightarrow 0$ uniformly on compact subsets.
\end{theorem}

We will use the Koebe $1/4$ theorem:

\begin{theorem}
Let $f$ be an analytic function on the unit disc. Suppose that $f(0)=0, f'(0)=1$ and $f$ is univalent, i.e. 1-1. Then $f(\{|z|<1\})\supset \{|z|< 1/4\}.$ 
\end{theorem}

We prove first  the snail lemma, and get back to talk about the proof of the Koebe 1/4 theorem in the end. So let $U$ be an invariant Fatou component so that $P^n$ converges on compact subsets to the fixed point $0.$ 

Pick a point $z_0\in U$ with orbit $\{z_n\}.$ Let $V_0\subset \subset U$ be a connected open set containing $ z_0$ and $z_1.$ We define $V_n=P^n(V_0)$. Then the union $\cup V_n$ is called a snail. It contains all the points of the orbit and it converges to $0$.

Since $P'(0)\neq 0$, the functions $P$ and $P^n$ are univalent in a neighborhood of $0$. Hence if we start later in the sequence, we can assume that all the $V_j$ are in this neighborhood and we can also shrink $V_0$ to be a topological disc. Define
$$
\phi_n(z)= \frac{P^n(z)}{P^n(z_0)}.
$$

Then these functions are univalent on $V_0$, $\phi_n(z_0)=1$ and the image of $V_0$ does not contain the origin. We let $\psi$ denote a biholomorphic map from the unit disc to $V_0$ such that
$\psi(0)=z_0.$ Define $h_n(\zeta) = \phi_n(\psi(\zeta))-1.$ Then $h_n$ is univalent on the unit disc, $h_n(0)=0$ and $h_n \neq -1$. The function $\frac{h_n(\zeta)}{h_n'(0)}$  satisfies the condition that
the derivative at $0$ is one and this function omits the value $\frac{-1}{h_n'(0)}. $ Therefore the Koebe $1/4$ theorem implies that $|\frac{-1}{h_n'(0)}|\geq \frac{1}{4}. $ Hence $|h_n'(z_0)|\leq 4.$ 
Note that the family of univalent functions $f$ on the unit disc with $f(0)=0$ and $f'(0)=1$ is normal. This is because their image must omit a point on the unit circle. Then if you restrict to any
subdisc of the unit disc which does not include the origin, you get a family of maps avoiding $0$ and some point on the unit circle. This implies normality on the punctured disc. The maximum principle extends normality to the disc.\\

Hence the family $\frac{h_n(\zeta)}{h_n'(0)}$ normal. Since $|h_n'(0)|\leq 4$ 
it follows that also the $h_n(\zeta)=\frac{h_n(\zeta)}{h_n'(0)}h_n'(0)$ is normal. Therefore also the family $\phi_n$ is normal on $V_0.$ 

Next we show that all limit functions $g$ on $V$ are univalent with $|g'(z_0)|\geq \delta > 0$ for a fixed delta. If not, there is a sequence where the derivative goes to $0$. Then the derivative must go to zero on all of $V_0$. So for some large $n$, we have the $\phi_n$ map $V_0$ into
a small sector around $1$ (after shrinking $V_0$ a little). The same it true for $P^n$. Since $P(z)=e^{2\pi i \theta}z+ 
\mathcal O(z^2),$ we see then the image $P^{n+1}(V_0)$ is disjoint from $P^n(V_0).$ This is impossible since both sets contain the point $z_{n+1}.$ It follows from the Koebe $1/4$ theorem that the sets $P^n(V_0)$ contain a disc of radius $\sigma |z_n|$ centered at $z_n$.

The following is a well known fact about irrational numbers.

\begin{lemma}
Let $\theta$ be an irrational number and let $\sigma>0.$ Then there exists an integer $N>1$
so that every $e^{2\pi i \psi}$ is closer to some $e^{2\pi i k \theta}, 0\leq k \leq N$ than $\sigma/4.$
\end{lemma}

It follows that for large $n,$ the points $z_n,z_{n+1}$ is inside an annulus
covered by $f^n(V_0)\cup f^{n+1}(V_0)\cup \cdots \cup f^{n+N}(V_0).$
Hence the set $\cup_n f^n(V_0)$ covers a punctured disc $0<|z|<\eta.$
Let $C$ denote the circle of radius $\eta>2.$ It follows that for $n$ large enough $f^n(C)$
is contained in the set $0<|z|< \eta/4.$

This implies that the derivative of $f^n$ at $0$ is at most $1/2.$ This contradicts
that $|f'(0)|=1.$ Hence the snail lemma follows.\\

Now we discuss the proof of the Koebe 1/4 theorem

\begin{theorem}
Let $f$ be an analytic function on the unit disc. Suppose that $f(0)=0, f'(0)=1$ and $f$ is univalent, i.e. 1-1. Then $f(\{|z|<1\})\supset \{|z|< 1/4\}.$ 
\end{theorem}

We need:

\begin{theorem} (Area theorem)
Let $g(z)=\frac{1}{z}+\sum_{n=0}^\infty b_n z^n$. Suppose $g$ is univalent as a map into the Riemann sphere. Then $\sum_{n=0}^\infty n|b_n|^2\leq 1.$
\end{theorem}

\begin{proof}
Fix $0<r<1.$ We set $D_r=\mathbb C \setminus g(\overline{\{|z|<r\}}.$ We use Stokes' theorem
to calculate the area of $D_r,|D_r|.$
\bea
|D_r| & = & \int \int_{D_r}dxdy \\
& = & \frac{1}{2i}\int \int_{D_r}  d\overline{w}\wedge dw\\
& = & \frac{1}{2i}\int \int_{D_r} d(\overline{w}dw)\\
& = & \frac{1}{2i} \int_{\partial D_r} \overline{w}dw\\
& = & \frac{-1}{2i} \int_{\partial \{|z|<r\}}\overline{g} dg\\
& = & \frac{-1}{2i}\int_{\partial \{|z|<r\} }\left(\frac{1}{\overline{z}}+\overline{b}_0
+\overline{b}_1 \overline{z}+\cdots\right)          
\left(\frac{-1}{z^2} +b_1+\cdots  \right)dz\\
\\
& = & \frac{-1}{2i}\int_{\partial \{|z|<r\} }\left(\frac{z}{r^2}+\overline{b}_0
+\overline{b}_1 \frac{r^2}{z}+\cdots\right)          
\left(\frac{-1}{z^2} +b_1+\cdots  \right)dz\\
& = & \frac{1}{2i}\int_{\partial \{|z|<r\} }\frac{1}{z}\left(\frac{1}{r^2}-|b_1|^2r^2 +\cdots           \right)dz\\
& = & \pi \left(\frac{1}{r^2}-\sum_{n=1}^\infty n|b_n|^2 r^{2n}\right)\\
\eea

Since the area is positive, it follows that
$$
\sum_{n=1}^\infty n|b_n|^2 r^{2n}\leq \frac{1}{r^2}
$$
\noindent for all $r<1.$
The result follows.
\end{proof}

\begin{corollary}
Let $f(z)=z+\sum_{n=2}a_n z^n$ be unitary on the unit disc. Then $|a_2|\leq 2.$
\end{corollary}

\begin{proof}
Let $h(z)=f(z^2).$ Then 
\bea
h(z) & = & z^2+\sum_{n=2}^\infty a_n z^{2n}\\
& = &z^2(1+\sum_{n=2}^\infty a_n z^{2n-2})\\
& = & \left(z\sqrt{1+\sum_{n=2}^\infty a_n z^{2n-2}}\right)^2\\
& = & z(1+\frac{1}{2} a_2 z^2+\cdots)\\
\eea
 Hence $h(z)$ has a welldefined square root on the unit disc of the form $z+\cdots$
and the square root is an odd function. Suppose that $h(z_1)=h(z_2).$ Then also
$f(z_1^2)=f(z_2^2)$ so $z_1^2=z_2^2.$ This implies that $z_2=\pm z_1.$ Since $h$ is
odd, we get that $z_1=z_2.$ Hence $h$ is 1-1 on the unit disc.
Next, let $g(z)=1/h(z)=\frac{1}{z}\frac{1}{1+\frac{1}{2}a_2 z^2 +\cdots }=
\frac{1}{z}-\frac{a_2}{2} z+\cdots.$
Following the area theorem we get that $|a_2|\leq 2.$
\end{proof}

We can now prove the Koebe 1/4 theorem. Let $f(z)=z+a_2z^2+\cdots.$ Pick a complex number $c$ and assume that $f(z)\neq c$ on the unit disc. Let $g(z)=\frac{cf(z)}{c-f(z)}.$ Then $g$ is holomorphic on the unit disc,
$g(0)=0, g'(0)=1. $ Also if $g(z_1)=g(z_2),$ then $f(z_1)=f(z_2)$ so $g$ is univalent.

\bea
g(z) & = & \frac{cf(z)}{c-f(z)}\\
& = & \frac{cz+ca_2 z^2+\cdots}{c-z-a_2 z^2+\cdots}\\
& = & \frac{1}{c} \frac{cz+ca_2 z^2+\cdots}{1-z/c-a_2/c  z^2+\cdots}\\
& = & \frac{1}{c}(cz+ca_2 z^2 +\cdots)(1+z/c+\cdots)\\
& = & (z+a_2 z^2 +\cdots)(1+z/c+\cdots)\\
& = & z+(a_2+\frac{1}{c})z^2+\cdots\\
\eea

Applying the area theorem to $f$ we get $|a_2|\leq 2$ and applying it to $g(z)$ we get
$|a_2+\frac{1}{c}|\leq 2.$ Hence,
$$
\frac{1}{c} \leq |a_2|+|a_2+\frac{1}{c}|\leq 4.
$$

So $|c|\geq \frac{1}{4}$ as we wanted to prove.
So we have finished the proof of the Koebe 1/4 theorem.\\


\begin{thebibliography}{}


\bibitem[BD]{BD} Briend, Jean-Yves; Duval, Julien; Deux caract\'erisations de la mesure d'\'equilibre d'un endomorphisme de Pk(C). (French) [Two characterizations of the equilibrium measure of an endomorphism of Pk(C)] Publ. Math. Inst. Hautes Études Sci. No. 93 (2001), 145–159.

\bibitem[FLM]{FLM} Freire, A., Lopes, A., Mane, R.: An invariant measure for rational maps. Bol. Soc. Brasil. Mat 14 (1983), 45-62.

\bibitem[L]{L} Lyubich, M.: Entropy properties of rational endomorphisms of the Riemann sphere. Ergodic Theory, Dynamical Systems, 3 (1983), 351-385.

\bibitem[M]{M}
Milnor, John; Dynamics in one complex variable. Third edition. Annals of Mathematics Studies, 160. Princeton University Press, Princeton, NJ, 2006. Preliminary version: arXiv:math/920173 is good enough for the course.


\end{thebibliography}
\end{document}